\documentclass[12pt]{amsart}
\usepackage{amsmath}
\usepackage{amsthm}
\usepackage{amsfonts,mathrsfs}
\usepackage{amssymb}
\usepackage{graphicx,hyperref}
\usepackage{enumitem}
\usepackage{color}
\usepackage[foot]{amsaddr}
\usepackage{verbatim}

\usepackage{lineno,xcolor}

\theoremstyle{plain}
\newtheorem{theorem}{Theorem}[section]
\newtheorem{result}{Result}[section]
\newtheorem{lemma}[theorem]{Lemma}

\newtheorem{corollary}[theorem]{Corollary}
\newtheorem{proposition}[theorem]{Proposition}

\theoremstyle{definition}
\newtheorem{conjecture}{Conjecture}

\newtheoremstyle{TheoremNum}
	{\topsep}{\topsep}              
  {\itshape}                      
  {}                              
  {\bfseries}                     
  {.}                             
  { }                             
  {\thmname{#1}\thmnote{ \bfseries #3}}
\newtheorem{remark}{Remark}

\newcommand{\R}{\mathbb R}

\newcommand{\Z}{\mathbb Z}

\newcommand{\RN}[1]{%
  \textup{\uppercase\expandafter{\romannumeral#1}}%
}
\newcommand{\rn}[1]{%
  \textup{\lowercase\expandafter{\romannumeral#1}}%
}

 \def\zhou#1 {\fbox {\footnote {\ }}\ \footnotetext { From Yue: {\color{red}#1}}}

\begin{document}
	\title[]{No lattice tiling of $\Z^n$ by Lee Sphere of radius 2}
	\author[K.\ H.\ Leung]{Ka Hin Leung\textsuperscript{\,1}}
	\address{\textsuperscript{1}Department of Mathematics, National University of Singapore, 119076 Singapore}
	\email{matlkh@nus.edu.sg}
	\author[Y.\ Zhou]{Yue Zhou\textsuperscript{\,2, $\dagger$}}
	\address{\textsuperscript{2}College of Liberal Arts and Sciences, National University of Defense Technology, 410073 Changsha, China}
	\address{\textsuperscript{$\dagger$}Corresponding author}
	\email{yue.zhou.ovgu@gmail.com}
	
	\begin{abstract}
		We prove the nonexistence of lattice tilings of $\Z^n$ by Lee spheres of radius $2$ for all dimensions $n\geq 3$. This implies that the Golomb-Welch conjecture is true when the common radius of the Lee spheres equals $2$ and $2n^2+2n+1$ is a prime. As a direct consequence, we also answer an open question in the degree-diameter problem of graph theory: the order of any abelian Cayley graph of diameter $2$ and degree larger than $5$ cannot meet the abelian Cayley Moore bound.
	\end{abstract}
	\keywords{Golomb-Welch conjecture; Lattice tiling; Algebraic tiling; Degree-diameter problem; Perfect Lee code.}
	\subjclass{52C22, 11H31, 05C25, 11H71}
	\maketitle

\section{Introduction}
The \emph{Lee distance} (also known as \emph{$\ell_l$-norm}, \emph{taxicab metric}, \emph{rectilinear distance} or \emph{Manhattan distance}) between two vectors $x=(x_1,x_2,\cdots, x_n)$ and $y=(y_1,y_2,\cdots, y_n)\in\Z^n$ is defined by
\[ d_L(x,y)=\sum_{i=1}^n|x_i-y_i|. \]
Let $S(n,r)$ denote the Lee sphere of radius $r$ centered at the origin in $\Z^n$, i.e.
\[S(n,r)=\left\{(x_1,\cdots, x_n)\in\Z^n: \sum_{i=1}^{n}|x_i|\leq r\right\}.  \]
If there exists a subset $C\in \Z^n$ such that $\mathcal{T}=\{S(n,r)+c: c\in C \}$ forms a partition of $\Z^n$, then we say that $\mathcal{T}$ is a \emph{tiling} of $\Z^n$ by $S(n,r)$. If $C$ is further a lattice, then we call $\mathcal{T}$ a \emph{lattice tiling}.

One may get a geometric interpretation of tilings of $\Z^n$ by Lee spheres in the following way. Let $\R$ denote the set of real numbers and $C(x_1, \cdots, x_n)=\{(y_1,\cdots, y_n): |y_i-x_i|\leq 1/2 \}$ which is the $n$-cube centered at $(x_1,\cdots, x_n)\in \R^n$. Let $L(n,r)$ be the union of $n$-cubes centered at each point in $S(n,r)$. Figure \ref{fig:L} depicts $L(n,r)$ for $n=2,3$ and $r=1,2$. 

\begin{figure}[h!]
	\centering
	\includegraphics[width=1\linewidth]{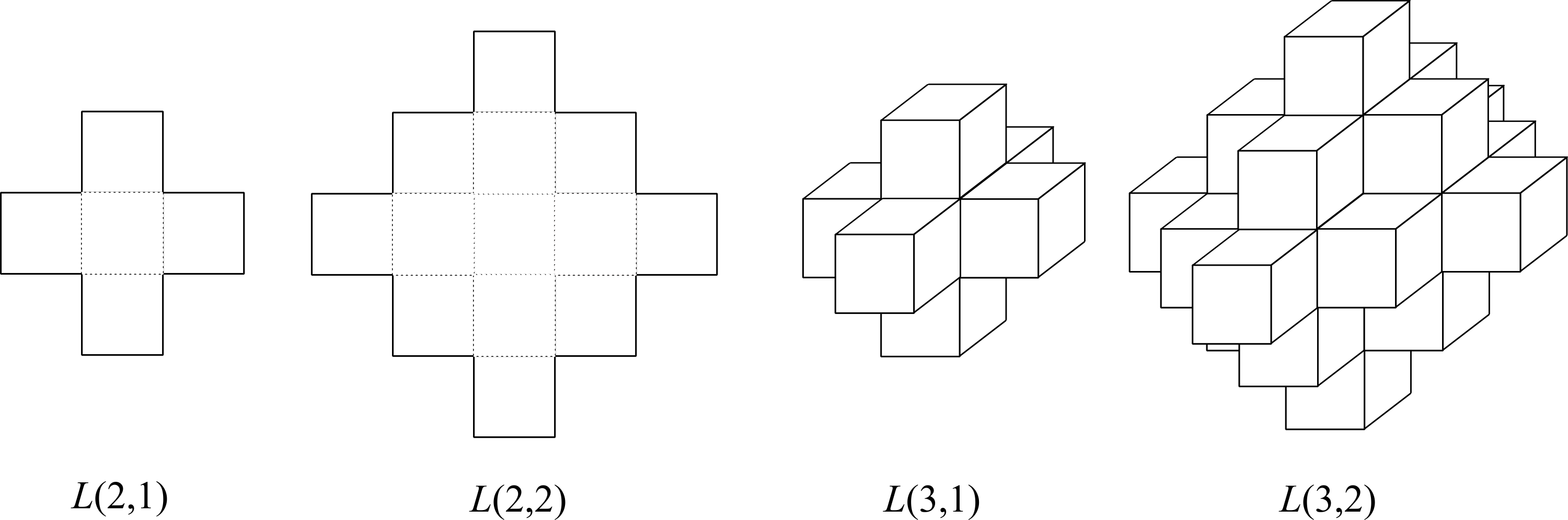}
	\caption{Figures of $L(2,1)$, $L(2,2)$, $L(3,1)$ and $L(3,2)$}
	\label{fig:L}
\end{figure}

It is easy to see that a tiling of $\Z^n$ by $S(n,r)$ exits if and only if a tiling of $\R^n$ by $L(n,r)$ exists. Figure \ref{fig:tiling} shows a (lattice) tiling of $\R^2$ by $L(2,2)$. Actually lattice tilings of $\R^n$ by $L(n,r)$ for $n=1,2$ and any radius $r$ always exist and lattice tilings of $\R^n$ by $L(n,1)$ also exist for any $n$; see \cite{golomb_perfect_1970}. 

\begin{figure}[h!]
	\centering
	\includegraphics[width=0.8\linewidth]{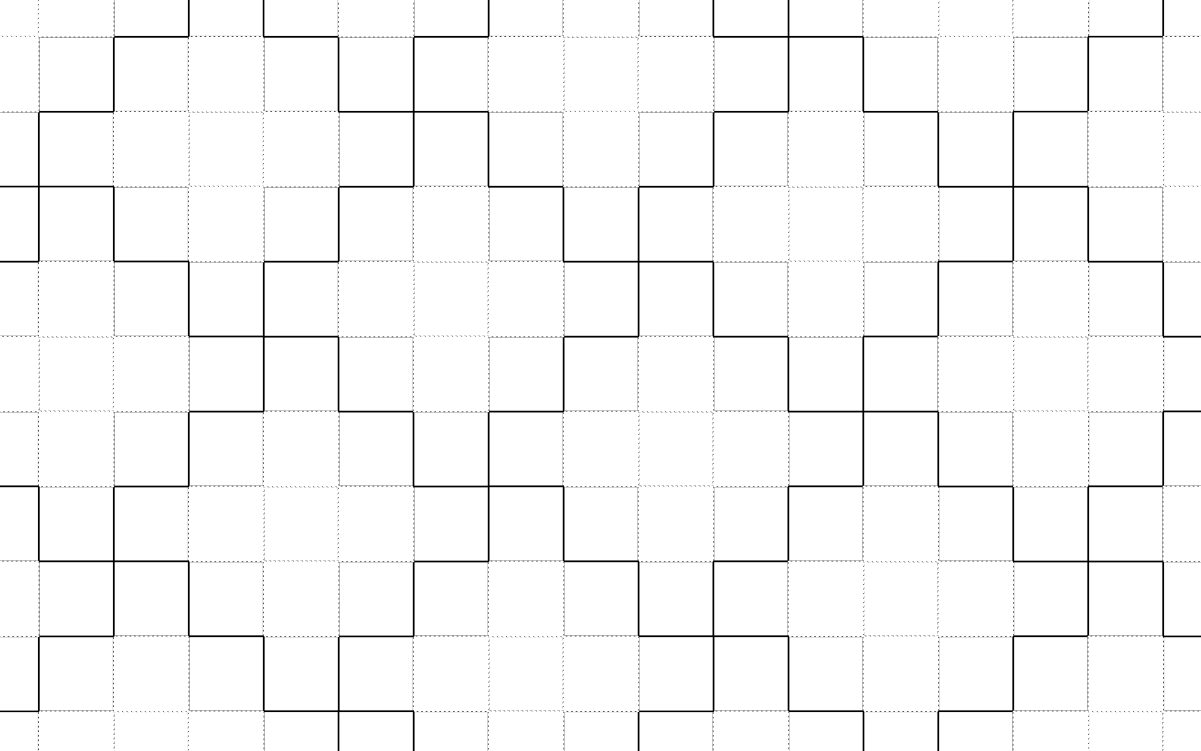}
	\caption{Tiling of $\R^2$ by $L(2,2)$}
	\label{fig:tiling}
\end{figure}

This geometric interpretation in $\R^n$ is quite important, because $L(n,r)$ is close to a cross-polytope when $r$ is large enough. It follows that a tiling of $\R^n$ with $L(n,r)$ induces a dense packing of $\R^n$ by cross-polytopes. One can use the cross-polytope packing density or the linear programming method which is originally applied on the Euclidean sphere packing density in \cite{cohen_sphere_packing_2003} to show the following type of results.
\begin{result}\label{result:geometric}
	For any $n\geq 3$, there exists $r_n$ such that for $r>r_n$, $\R^n$ cannot be tiled by $L(n,r)$. 	
\end{result}
Result \ref{result:geometric} was first obtained by Golomb and Welch who showed in \cite{golomb_perfect_1970} only the existence of $r_n$. However, the value of $r_n$  three is unspecified. Later, several lower bounds on $r_n$ for the periodic case were obtained by Post \cite{post_nonexistence_1975} and Lepist\"o \cite{lepisto_modification_1981}.  In \cite{horak_50_2018} the very first lower bound on $r_n$ is stated.

In the same seminal paper \cite{golomb_perfect_1970}, Golomb and Welch proposed the following conjecture originally given in the language of perfect Lee codes.
\begin{conjecture}\label{conj:GW}
	For $n\geq 3$ and $r\geq 2$, there is no perfect $r$-error-correcting Lee code in $\Z^n$, i.e.\ $\Z^n$ cannot be tiled by Lee spheres of radius $r$.
\end{conjecture}

Conjecture \ref{conj:GW} is still far from being solved, though various approaches have been applied on it. We refer the reader to the recent survey \cite{horak_50_2018} and the references therein.

In \cite{horak_50_2018} Horak and Kim suggest that $r=2$ appears to be the most difficult case of Conjecture \ref{conj:GW} for two reasons. First it is the threshold case, because $\Z^n$ can always be tiled by $S(n,1)$. Second the proof of Conjecture \ref{conj:GW} for $3 \leq n \leq 5$ and all $r \geq 2$ in \cite{horak_tilings_2009} is based on the nonexistence of tilings of $\Z^n$ by $S(n,2)$ for the given $n$.

In this direction, there are several recent advances. In \cite{horak_new_2014}, Conjecture \ref{conj:GW} is proved for $n\leq 12$ and $r=2$. In \cite{kim_2017_nonexistence}, Kim presents a method based on symmetric polynomials to show that Conjecture \ref{conj:GW} is true for $r=2$ and a certain class of $n$ satisfying that $|S(n,2)|$ is a prime. This approach has been further applied to the lattice tilings of $\Z^n$ by $S(n,r)$ with larger $r$ in \cite{qureshi_nonexistence_ZGcondition_arxiv} and \cite{zhang_perfect_lp_2007}. In \cite{zhang_nonexistence_2019}, Zhang and the second author translated the lattice tilings of $\Z^n$ by $S(n,2)$ or $S(n,3)$ into group ring equations. By applying group characters and algebraic number theory, they have obtained more nonexistence results for infinitely many $n$ with $r=2$ and $3$.

In this paper, we completely solve the lattice tiling cases of Conjecture \ref{conj:GW} for $r=2$ and any $n$.
\begin{theorem}\label{th:main}
	For any integer $n\geq 3$, there is no lattice tiling of $\Z^n$ by $S(n,2)$.
\end{theorem}

It is worth noting that, in contrast to Result \ref{result:geometric} which is proved for fixed dimension $n$, Theorem \ref{th:main} is for fixed radius $r$ and arbitrary $n$.

It is straightforward to show that $|S(n,2)|=2n^2+2n+1$. According to \cite[Theorem 28]{szegedy_algorithms_1998} (see \cite[Exampe 2]{kari_algebraic_2015} for an alternative proof), when $2n^2+2n+1$ is a prime, a tiling of $\Z^n$ by $S(n,2)$ must be a lattice tiling. Thus Theorem \ref{th:main} implies the following result.
\begin{corollary}
	For $r=2$ and $n\geq 3$ satisfying that $2n^2+2n+1$ is prime, the Golomb-Welch conjecture is true.
\end{corollary}

Our result also answers an important question in graph theory. The \emph{degree-diameter} problem is to determine the largest graph of given maximum degree $d$ and diameter $k$. For the general case, the famous \emph{Moore bound} is an upper bound for the orders of such graphs. Except for $k = 1$ or $d \leq 2$, graphs achieving the Moore bound are only possible for $d = 3, 7, 57$ and $k = 2$; see \cite{bannai_1973_moore}~\cite{damerell_moore_1973} and \cite{hoffman_moore_1960}.

Let $G$ be a multiplicative group with the identity element $e$ and $S\subseteq G$ such that $S^{-1}=S$ and $e\not\in S$. Here $S^{-1}=\{s^{-1}: s\in S\}$. The (undirected) \emph{Cayley graph} $\Gamma(G,S)$ has a vertex set $G$, and two distinct vertices $g,h$ are adjacent if and only if $g^{-1}h\in S$. In particular, when $G$ is abelian, we call $\Gamma(G,S)$ an \emph{abelian Cayley graph}. 

Let $AC(d,k)$ denote the largest order of abelian Cayley graphs of degree $d$ and diameter $k$. In \cite{dougherty_degree-diameter_2004}, an upper bound for $AC(2n,r)$ is obtained which actually equals 
\[|S(n,r)|=\sum_{i=0}^{\text{min}\{n,r\}}2^{i}\binom {n}{i}\binom {r}{i}.\]
This value is often called the \emph{abelian Cayley Moore bound}.
An important open question in graph theory is whether there exists an abelian graph whose order meets this bound. For more details about the degree-diameter problems, we refer to the survey \cite{miller_moore_2013}.

By checking the proof of the upper bound for $AC(2n,r)$ in \cite{dougherty_degree-diameter_2004}, it is not difficult to see that an abelian Cayley graph of degree $2n$ and diameter $r$ achieves this upper bound if and only if there is a lattice tiling of $\Z^n$ by $S(n,r)$; see \cite[Section 2.1]{zhang_nonexistence_2019} for the detail. This link is also pointed out in \cite{camarero_quasi-perfect_lee_2016}. Hence, Theorem \ref{th:main} is equivalent to the following statement.
\begin{corollary}
	The number of vertices in any abelian Cayley graph of diameter $2$ and even degree $d\geq 6$ is strictly less than the abelian Cayley Moore bound.
\end{corollary}

The rest of this paper is organized as follows: In Section \ref{sec:pre}, we introduce the group ring conditions for the existence of a lattice tiling of $\Z^n$ by $S(n,2)$. In Section \ref{sec:main}, we prove Theorem \ref{th:main}.

\section{Preliminaries}\label{sec:pre}
Let $\Z[G]$ denote the set of formal sums $\sum_{g\in G} a_g g$, where $a_g\in \Z$ and $G$ is any (not necessarily abelian) group which we write here multiplicatively. The addition of elements in $\Z[G]$ is defined componentwise, i.e.
$$\sum_{g\in G} a_g g +\sum_{g\in G} b_g g :=\sum_{g\in G} (a_g+b_g) g.$$
The multiplication is defined by
$$(\sum_{g\in G} a_g g )\cdot (\sum_{g\in G} b_g g) :=\sum_{g\in G} (\sum_{h\in G} a_hb_{h^{-1}g})\cdot g.$$
Moreover,
$$\lambda \cdot(\sum_{g\in G} a_g g ):= \sum_{g\in G} (\lambda a_g) g $$
for $\lambda\in\Z$. For $A=\sum_{g\in G} a_g g$ and $t\in \Z$, we define 
$$A^{(t)}:=\sum_{g\in G} a_g g^t.$$
For any set $A$ whose elements belong to $G$ ($A$ may be a multiset), we can identify $A$ with the group ring element $\sum_{g\in G} a_g g$
where $a_g$ is the multiplicity of $g$ appearing in $A$. Moreover, we use $|A|$ to denote the number of distinct elements in $A$, rather than the counting of elements with multiplicity.

The existence of a lattice tiling of $\Z^n$ by $S(n,2)$ can be equivalently given by a collection of group ring equations.
\begin{lemma}[\cite{zhang_nonexistence_2019}]\label{lm:group_ring}
	Let $n\ge2$. There exists a lattice tiling of $\Z^n$ by $S(n,2)$ if and only if there exists a finite abelian group $G$ of order $2n^{2}+2n+1$ and a subset $T$ of size $2n+1$ viewed as an element in $\mathbb{Z}[G]$ satisfying
	\begin{enumerate}[label=(\alph*)]
		\item\label{cond:1} the identity element $e$ belongs to $T$,
		\item\label{cond:2} $T=T^{(-1)}$,
		\item\label{cond:3} $T^{2}= 2G-T^{(2)} +2n$.
	\end{enumerate}
\end{lemma}

We also need the following nonexistence results summarized in \cite{zhang_nonexistence_2019}.
\begin{lemma}\label{lm:list<=100}
	For $3\leq n\leq 100$, there is no lattice tiling of $\Z^n$ by $S(n,2)$ except possibly for $n=16, 21, 36, 55, 64, 66, 78, 92$. 
\end{lemma}

\section{Proof of the main result}\label{sec:main}
Our objective is to show the nonexistence of $T$ satisfying Conditions \ref{cond:1}--\ref{cond:3} in Lemma \ref{lm:group_ring}. 
To do so, we do assume such $T$ exists and try to deduce some necessary consequences. 

The outline of our proof of Theorem \ref{th:main} is as follows: we first investigate $T^{(2)}T \pmod{3}$, which provides us some strong restrictions on the multiplicities of elements in $T^{(2)}T$. In particular, it leads to a proof of Theorem \ref{th:main} when $n\equiv 0 \pmod{3}$; see Proposition \ref{prop:n=0mod3}. Then we further look at $T^{(4)}T\pmod{5}$. For each of the rest $10$ possible value of $n$ modulo $15$, we can get a contradiction.

First, by using Condition \ref{cond:3}, we immediately obtain the following:

\begin{lemma}\label{lm:TmeetT2}
For any $g\in G\setminus \{e\}$, 
\[ |\{ (t_1,t_2)\in T\times T: t_1t_2=g \}|=\begin{cases}
		1, & \mbox{ if } g\in T^{(2)};\\
		2, & \mbox{ if } g\in G\backslash T^{(2)}
	\end{cases}.\]
In particular, $T\cap T^{(2)}=\{e\}$. Moreover, if $t, t_1,t_2\in T$, then $t_1t_2=t^2$ if and only if $t_1=t_2=t$. 
\end{lemma}
\medskip

Observe that by \ref{cond:3}, $	T^3 = 2(2n+1)G-T^{(2)}T+2nT$ which means
\begin{equation}\label{eq:T^2T}
	T^{(2)}T = 2(2n+1)G-T^3+2nT.
\end{equation}
Our strategy is to exploit the above equation. For convenience, we keep the following notation through this section. We write
\[ T^{(2)}T = \sum_{i=0}^{N} iX_i\]
where $\{X_i: i=0,1,\dots, N\}$ forms a partition of $G$. It is easy to deduce the following:
\begin{equation}\label{eq:main_2n^2+2n+1}
	2n^2+2n+1 = \sum_{i=0}^{N} |X_i|,
\end{equation}
and 
\begin{equation}\label{eq:main_(2n+1)^2}
	(2n+1)^2 = \sum_{i=1}^N i|X_i|. 
\end{equation}
Note that $|G|$ is odd and $|T^{(2)}|=2n+1$. Moreover, as $T^{(2)}\cap T=\{e\}$, it follows that $e\in X_1$. Besides the above two equations on $|X_i|$'s, we derive another equation based on the inclusion-exclusion principle as follows:

\begin{lemma}\label{lm:T^2T}
	\begin{equation}\label{eq:main_T2T}
		\sum_{i=1}^N|X_i| = 4n+1 + \sum_{s=3}^{N}\frac{(s-1)(s-2)}{2}|X_s|.
	\end{equation}
\end{lemma}
\begin{proof}
	By \ref{cond:1} and \ref{cond:2}, we can write $T^{(2)}= \sum_{i=0}^{2n}a_i$ with $a_0=e$ and $a_{i}^{-1}=a_{2n+1-i}$ for $i=1,2,\dots, 2n$. Clearly all the $a_i$'s are distinct from each other and
	\begin{equation}\label{eq:T2T_partition}
		T^{(2)}T=\sum_{i=0}^{2n}a_iT.
	\end{equation}
	First, we prove the following claim.
	
	\textbf{Claim 1.} For $0\leq i<j\leq 2n$, $|a_iT\cap a_jT|=
	\begin{cases}
		1, & 0=i<j;\\
		2, & 0<i<j.
	\end{cases}$

	Observe that $a_it=a_jt'\in a_iT\cap a_jT$ if and only if $a_i a_j^{-1}=t^{-1}t'$ for some $t,t'\in T$. Recall that $T^{(2)}$ also satisfies Condition \ref{cond:3} in Lemma \ref{lm:group_ring}. Hence, $a_i a_j^{-1}\notin T^{(2)}$ if and only if $i\neq 0$.

	If $a_ia_j^{-1} \in G\setminus T^{(2)}$, then by Lemma \ref{lm:TmeetT2}, there exist two distinct $t_1, t_2\in T$ such that $(t^{-1}, t')=(t_1,t_2)$ or $(t_2,t_1)$. Hence, $t=t_1^{-1}$ or $t_2^{-1}$. Consequently, $|a_iT\cap a_jT|=2$.  On the other hand, if $i=0$, $a_ia_j^{-1}=e a_j^{-1}=s^2\in T^{(2)}$ for some $s\in T$. By Lemma \ref{lm:TmeetT2}, $t^{-1}=s=t'$. Hence, $|a_iT\cap a_jT|=1$. 
	
	By the inclusion--exclusion principle and  \eqref{eq:T2T_partition}, we count the distinct elements in $T^{(2)}T$,
	\begin{equation}\label{eq:expand_T2T}
		|T^{(2)}T| = \sum_{i=0}^{2n} |a_iT| - \sum_{i<j} |a_iT \cap a_jT| + \sum_{r\geq 3} (-1)^{r-1} |a_{i_1}T\cap a_{i_2}T \cap \cdots \cap a_{i_r}T|,
	\end{equation}
	where $i_1<i_2<\cdots <i_r$ cover all the possible values. By definition of $X_i$'s, its left-hand side equals $\sum_{i=1}^N|X_i|$.
	
	It is clear that
	\begin{equation}\label{eq:sum_Li}
		\sum_{i=0}^{2n} |a_iT|=(2n+1)^2.
	\end{equation}
	By Claim 1, 
	\begin{equation}\label{eq:sum_LiLj}
		\sum_{i<j} |a_iT \cap a_jT| = 2n + 2\binom{2n}{2}=4n^2.
	\end{equation}
	Suppose that $g\in a_{i_1}T\cap a_{i_2}T \cap \cdots \cap a_{i_r}T$ with $r\geq 3$. It means that $g\in X_s$ for some $s\geq 3$. Then the contribution for $g$ in the sum $\sum_{r\geq 3} (-1)^{r-1} |a_{i_1}T\cap a_{i_2}T \cap \cdots \cap a_{i_r}T|$ is
	\[ (-1)^{3-1} \binom{s}{3}+(-1)^{4-1} \binom{s}{4} +\cdots =\binom{s}{2} -\binom{s}{1} + \binom{s}{0} = \frac{(s-1)(s-2)}{2}.\]
	Therefore,
	\[\sum_{r\geq 3} (-1)^{r-1} |a_{i_1}T\cap a_{i_2}T \cap \cdots \cap a_{i_r}T| = \sum_{s\geq 3} |X_s|\frac{(s-1)(s-2)}{2}.\]
	Plugging the above equation, \eqref{eq:sum_Li} and \eqref{eq:sum_LiLj} into \eqref{eq:expand_T2T}, we obtain \eqref{eq:main_T2T}.
\end{proof}

\medskip

Our strategy is to derive a contradiction using Equation (2), (3) and (4). We need to further exploit \eqref{eq:T^2T}. It is natural to consider \eqref{eq:T^2T} modulus $3$ as $T^3\equiv T^{(3)} \pmod{3}$. We then have 
\begin{equation}\label{eq:mod3}
	T^{(2)}T = 2(2n+1)G-T^{(3)}+2nT \pmod{3}. 
\end{equation}
Note that $|G|=2n^2+2n+1$ is not divisible by $3$. Therefore, $|T^{(3)}|=2n+1$. We first investigate the case when $n\equiv 0 \pmod{3}$.

\begin{proposition}\label{prop:n=0mod3}
	Theorem \ref{th:main} is true for $n\equiv 0\pmod 3$. 
\end{proposition}
\begin{proof} 
	Now \eqref{eq:mod3} becomes
	\begin{equation*}
	T^{(2)}T \equiv 2G-T^{(3)}\pmod{3}.
	\end{equation*}
	Since all coefficients are non-negative, $|T^{(3)}|=2n+1$ and all coefficients of $T^{(3)}$ is $1$, 
	we conclude that 
	\begin{equation}\label{eq:sum_x3i+j_nmod3=0}
	\sum_{i=0}|X_{3i+1}| = 2n+1, \quad	\sum_{i=0}|X_{3i+2}| = 2n^2 \quad \text{and}\quad \sum_{i=0}|X_{3i}| = 0.
	\end{equation}
	By \eqref{eq:main_(2n+1)^2} and \eqref{eq:sum_x3i+j_nmod3=0},
	\begin{equation}\label{eq:balance_1}
		2n= \sum_{i=1} 3i (|X_{3i+1}|+|X_{3i+2}|).
	\end{equation}
	We recall that $N$ is the largest integer with $|X_N|\neq 0$. By \eqref{eq:main_2n^2+2n+1} and \eqref{eq:main_T2T}, we have
	\begin{equation}\label{eq:balance_2}
		2n^2-2n=\sum_{s=4}^N\frac{(s-1)(s-2)}{2}|X_s| \leq 	
		\frac{N-1}{2} \sum_{s=4}^N (s-2) |X_s|. 
	\end{equation}
As $T^{(-1)}=T$ and $T^{(-2)}=T^{(2)}$, it is clear that $X_N=X_N^{(-1)}$. Recall that $e\in X_1$. It then follows that $|X_N|\geq 2$. From \eqref{eq:balance_1}, we derive that $2(N-2)\leq 2n$. 

By \eqref{eq:balance_1}, $N\equiv 1, 2\pmod{3}$ and that $\sum_{i=0}|X_{3i}|=0$, we conclude $\sum_{s=4}^N (s-2) |X_s|\leq 2n$. 
Therefore, we obtain from \eqref{eq:balance_2} that $ 2n^2-2n\leq  (n+1)n$. This is possible only if $n\leq 3$. As for $n=3$, it is already known in Lemma \ref{lm:list<=100} that Theorem \ref{th:main} is true. 
\end{proof}

\medskip

Unfortunately, using the above argument in case $n\equiv \pm 1 \pmod{3}$ does not rule the existence out.  But we are still able to obtain some essential informations in those cases.  

\begin{lemma}\label{lm:n=1mod3}
	Suppose $n\equiv 1\pmod 3$. 	Then, $|X_i|=0$ for all $i\geq 4$. Moreover, we have 
	\begin{equation}\label{eq:value_Xi_nmod3=1}
		|X_3| = \frac{4n(n-1)}{3}, \quad |X_0|=\frac{2n(n-1)}{3}, \quad |X_2|=4n \text{ and } |X_1|=1.
	\end{equation}
\end{lemma}
\begin{proof} 
	Now \eqref{eq:mod3} becomes
	\begin{equation}\label{eq:T2T_nmod3=1}
		 T^{(2)}T \equiv 2T-T^{(3)} \pmod{3}.
	\end{equation}
	We first show that $T\cap T^{(3)}=\{e\}$. Suppose that $t$ and $t_0\in T$ satisfy that $t=t_0^3\in T$. Then $t t_0^{-1} = t_0^2\in T^{(2)}$ which means $t=t_0^{-1}=t_0$ by \ref{cond:3}. Hence $t=t_0=e$ because $|G|$ is odd.
	
	By comparing the coefficients in \eqref{eq:T2T_nmod3=1}, we see that except for those elements in $T\cup T^{(3)}$, all are congruent to $0 \bmod 3$. Since $T\cap T^{(3)}=\{e\}$ and $e\in X_1$, the coefficients of all the elements in $T\cup T^{(3)}\backslash \{e\}$ are congruent to $2 \bmod 3$. Therefore, 
	we get 
	\[ |X_1|=1, \ |X_{3i+1}|=0  \mbox{ for } i\geq 1 \mbox{ and } \sum_{i=0}|X_{3i+2}|=4n.\]
	Plugging them into \eqref{eq:main_T2T}, we get
	\[ 1+|X_2| +\sum_{i=3}^N|X_i| = \sum_{i=0}|X_{3i+2}|+1 + \sum_{s=3}^{N}\frac{(s-1)(s-2)}{2}|X_s|,\]
    from which we conclude that 
	\[ 2|X_4| + 4|X_5| +9|X_6| +\cdots \leq 0. \]
	This implies that $|X_i|=0$ for $i\geq 4$ and $|X_2|=4n$. Hence, by \eqref{eq:main_2n^2+2n+1} and \eqref{eq:main_(2n+1)^2},
	\[\left\{
	\begin{array}{l}
	2n^2+2n+1 = |X_0|+1+4n+|X_3|,\\
	4n^2+4n+1 = 1+2\cdot 4n +3|X_3|.
	\end{array}
	\right.\]
	Solving the above equations, we get the desired result. 
	\end{proof}	
	
	\medskip
	
	Next, we consider the case when $n \equiv 2 \pmod{3}$. 
	
\begin{lemma}\label{lm:n=2mod3}
	Suppose $n\equiv 2\pmod 3$. Then, $|X_i|=0$ for all $i\geq 5$. Moreover, we have 
	\begin{equation}\label{eq:value_Xi_nmod3=2}
		|X_1| = \frac{4n^2-2n+3}{3},\quad |X_2|= |X_3|=2n,\quad |X_4|=\frac{2n^2-4n}{3}\text{ and }|X_0|=0.
	\end{equation}
\end{lemma}
\begin{proof}
	 Now \eqref{eq:mod3} becomes
	\begin{equation*}
	T^{(2)}T \equiv G-T^{(3)}+T \pmod{3}.
	\end{equation*}
	As shown before, we have $T\cap T^{(3)} =\{e\}$. Thus, $X_0=\emptyset$ and
	\[ T^{(3)}\backslash \{e\} =\bigcup_{i=0} X_{3i},\  T\backslash \{e\} =\bigcup_{i=0} X_{3i+2} \mbox{ and } G\backslash (T\cup T^{(3)}) \cup \{e\}=\bigcup_{i=0} X_{3i+1}.\]
	
	Hence
	\begin{equation}\label{eq:sum_x3i+1_nmod3=2}
	\sum_{i=0}|X_{3i+1}| = 2n^2-2n+1,
	\end{equation}
	\begin{equation}\label{eq:sum_x3i+2and0_nmod3=2}
	\sum_{i=0}|X_{3i+2}| = 2n \quad \text{and}\quad \sum_{i=0}|X_{3i}| = 2n.
	\end{equation}
	By \eqref{eq:main_T2T}, we have
	\[2n^2+2n+1 \geq 4n+1 + \sum_{s=3}^{N}\frac{(s-1)(s-2)}{2}|X_s|.\]
	Summing up the above equation and \eqref{eq:sum_x3i+1_nmod3=2}, we have 
	\[4n^2-4n+1 \geq |X_1| + |X_3| + 4|X_4| +6|X_5| + 10|X_6|+16|X_7|+\cdots.\]
	
	On the other hand, plugging \eqref{eq:main_(2n+1)^2}, \eqref{eq:sum_x3i+2and0_nmod3=2} and $X_0=\emptyset$ into $\sum_{i=1}i|X_i|-2\sum_{i=0} |X_{3i}|-2\sum_{i=0} |X_{3i+2}|$, we have
	\[ 4n^2-4n+1=|X_1|+|X_3| + 4|X_4| + 3|X_5| + 4|X_6| + 7|X_7| +\cdots.  \]

	Thus $|X_5|=|X_6|=\cdots =0$. It follows from \eqref{eq:sum_x3i+2and0_nmod3=2} that $|X_2|=2n$. By solving
	\[
	\left\{
	\begin{array}{ll}
	4n^2-4n+1&=|X_1|+|X_3| + 4|X_4|,\\
	4n^2+4n+1&=|X_1|+2|X_2| + 3|X_3| + 4|X_4|,\\
	2n^2-2n+1 &= |X_1|+|X_4|.
	\end{array}
	\right.
	\]
	we get our desired result.  
\end{proof}
	
	\medskip
	
In view of the above results, we see that by just considering modulus $3$, it doesn't rule out the case for $n\equiv 1,2 \pmod{3}$.  It is natural to consider a similar equation modulus $5$. By \ref{cond:3}, 
\[ T^5 = (8n^2+8n+2)(2n+1)G - T^{(4)}T -4n T^{(2)}T + (4n^2+2n)T\]
which implies
\begin{equation}\label{eq:T^4T}
	T^{(4)}T = (8n^2+8n+2)(2n+1)G - T^5 -4n T^{(2)}T + (4n^2+2n)T.
\end{equation}
As before, we write 
\[ T^{(4)}T = \sum_{i=0}^{M} iY_i \]
where $\{Y_i: i=0,1,\dots, M\}$ forms a partition of $G$.
Since $|T^{(4)}|=2n+1$, we have :
\begin{equation*}
	2n^2+2n+1 = \sum_{i=0}^{M} |Y_i|,
\end{equation*}
and 
\begin{equation}\label{eq:Y_is_condition2}
	(2n+1)^2 = \sum_{i=1}^M i|Y_i|. 
\end{equation}
However, the situation is slightly different now. 
\begin{lemma}
	There exists an integer $\Delta\in [-2n,0]$ such that
	\begin{equation}\label{eq:main_T4T}
	\sum_{i=1}^M|Y_i| = 4n+1 +\Delta + \sum_{s=3}^{M}\frac{(s-1)(s-2)}{2}|Y_s|.
	\end{equation}
	Moreover, we have
	\begin{equation}\label{ok1}
		2|Y_1| + 3|Y_2| +3|Y_3| +2 |Y_4| \geq 4n^2+6n+2.
	\end{equation}
\end{lemma}
\begin{proof}
	The proof is quite similar to the one for Lemma \ref{lm:T^2T}. The only different part is Claim 1.
	
	By \ref{cond:1}, we may write $T^{(4)}= \sum_{i=0}^{2n}a^2_i$ with $a_0=e$ and $a_{i}^{-1}=a_{2n+1-i}$ for $i=1,2,\dots, 2n$. Hence
	\begin{equation*}
	T^{(4)}T=\sum_{i=0}^{2n}a^2_iT.
	\end{equation*}
	
	\textbf{Claim 1.} $|T\cap a^2_iT|=2$ for $i>0$ and 
	\[ 4n^2- 6n\leq\sum_{0<i<j\leq 2n+1} |a_i^2T\cap a_j^2T|\leq 4n^2-4n.\]
	
	Suppose that $t\in T\cap a^2_iT$. Then there exists $t_0$ such that $tt_0 = a_i^2\in T^{(4)}\setminus\{e\}$. Note that $a_i^2\in G\setminus T^{(2)}$ as $T^{(4)}\cap T^{(2)}=\{e\}$ by \ref{cond:3}. Hence, there are two choices for $t$ and hence $|T\cap a^2_iT|=2$. 
	
	For $j>i>0$, as shown before, 
	\[ |a_i^2T\cap a_j^2T| = \begin{cases}
		1, &a_i^{-2}a_j^2\in T^{(2)},\\
		2, & \text{ otherwise.}
	\end{cases} \]
	To find the number of pairs $(i,j)$ with $0<i<j$ when  $|a_i^2T\cap a_j^2T| =1$, we need to find for each $s\in T^{(2)}$, the number of pairs of $(i,j)$ with $i<j$ and $a_i^{-2}a_j^2=s$. Since $a_i^{-2}, a_j^2\in T^{(4)}$ and $T^{(4)}$ also satisfies Condition \ref{cond:3}, it follows that the number of pairs is at most $1$. Therefore, 
	\[0\leq \delta =|\{(i,j): 0<i<j,~~a_i^{-2}a_j^2\in T^{(2)}\}|\leq 2n. \]
	Thus
	\begin{align*}
	\sum_{i<j} |a^2_iT \cap a^2_jT| = &\sum_{i\neq 0} |T \cap a^2_iT|+\sum_{0<i<j} |a^2_iT \cap a^2_jT|\\
	=&2\cdot 2n+ 2\cdot(\binom{2n}{2}-\delta)+\delta\\
	=& 4n^2+2n-\delta.
	\end{align*}
	By applying a similar argument to $T^{(4)}T$ as in the proof of in Lemma \ref{lm:T^2T}, we obtain 
	\[\sum_{i=1}^M|Y_i| = (2n+1)^2 -(4n^2+2n-\delta) + \sum_{s=3}^{M}\frac{(s-1)(s-2)}{2}|Y_s|.\]
	Setting $\Delta= \delta-2n$, we obtain \eqref{eq:main_T4T}.
	
	By adding up \eqref{eq:main_T4T} and \eqref{eq:Y_is_condition2}, we obtain 
	\begin{equation*}
		 2|Y_1| + 3|Y_2| +3|Y_3| +2 |Y_4| = 4n^2+8n+2+\Delta + \sum_{s=5}(\frac{(s-1)(s-2)}{2}-s-1)|Y_s|.
	\end{equation*}
	 Finally, as the last sum is always non-negative and $\Delta\geq -2n$, we obtain \eqref{ok1}. 
\end{proof}

\medskip 

Now, we are ready to resume the proof of Theorem \ref{th:main}. First we consider a special case.
\medskip

\begin{proposition}\label{prop:nmod5=0}
	Theorem \ref{th:main} is true for $n\equiv 0\pmod{5}$.
\end{proposition}
\begin{proof}
	By \eqref{eq:T^4T}, we obtain 
	\begin{equation}\label{0mod5}	
		T^{(4)}T \equiv 2G-T^{(5)} \pmod{5}.
	\end{equation}
	In this case, $5$ doesn't divide $|G|$. Therefore, $|T^{(5)}|=2n+1$. Consequently, 
	\[   \bigcup_{i=0} Y_{5i+1}=T^{(5)}, \ \bigcup_{i=1} Y_{5i+2}=G\backslash T^{(5)}.\]
	Therefore,  $|Y_i|=0$ for $i\not\equiv 1,2 \pmod{5}$ and 
	\begin{equation}\label{eq:Y_5i+1_andY_5i+2}
		\sum_{i=0}|Y_{5i+1}| = 2n+1,\quad \sum_{i=0}|Y_{5i+2}| = 2n^2. 
	\end{equation}
		Hence
	\begin{equation}\label{eq:sum_Y_1+5i_+_2Y_2+5i}
		 \sum_{i=0}|Y_{5i+1}|+ \sum_{i=0}2|Y_{5i+2}| = 4n^2 + 2n +1. 
	\end{equation}
	On the other hand,
	\[|Y_1| + 2|Y_2| + 6|Y_6| + 7 |Y_7| + \cdots=\sum_{i=1}^M i|Y_i|=4n^2+4n+1. \]
	Together with \eqref{eq:sum_Y_1+5i_+_2Y_2+5i}, we get
	\begin{equation}\label{eq:5i_Y5i+1_Y5i+2}
		 5\sum_{i=1}i(|Y_{5i+1}|+|Y_{5i+2}|)=2n. 
	\end{equation}
	Recall that $M=\max\{ i: Y_i\neq \emptyset\}$.  By \eqref{eq:main_T4T} and \eqref{eq:Y_5i+1_andY_5i+2}, 
	\begin{equation}\label{eq:nmod5=0_balance_2}
		2n^2-2n-\Delta = \sum_{s=3}^M \frac{(s-1)(s-2)}{2} |Y_s| 	 \leq 	
		\frac{M-1}{2} \sum_{s=3}^M (s-2) |Y_s|. 
	\end{equation}
	As $|Y_i|=0$ for $i\not\equiv 1,2 \pmod{5}$, it follows from \eqref{eq:5i_Y5i+1_Y5i+2}
	\begin{equation}\label{eq:nmod5=0_balance_3}
		 \sum_{s=3}^M  (s-2)|Y_s|\leq 2n.
	\end{equation}

	\textbf{Case (\rn{1})} If $|Y_M|\geq 2$, then as in the proof of Proposition \ref{prop:n=0mod3}, we obtain $2(M-2)\leq 2n$ and 
	$M-1\leq n+1$. Plugging them into \eqref{eq:nmod5=0_balance_2} and \eqref{eq:nmod5=0_balance_3}, we get $2n^2-2n-\Delta  \leq (n+1)n$ and $n\leq 3$. This is impossible. 

	\textbf{Case (\rn{2})} $|Y_M|=1$.  Note that $Y_M=Y_M^{-1}$. Hence, $Y_M=\{e\}$.

	If $M<2n+1$, then $M-1<2n$. Then,  in view of \eqref{eq:5i_Y5i+1_Y5i+2}, there exists $j\neq M$ such that $|Y_j|\geq 1$. Suppose $j=5i+c$ where $i\geq 1$ and $c=1$ or $2$. Again, $Y_j=Y_j^{-1}$ implies, $M-2\leq 2n-10$. Consequently, 
	\[2n^2-2n-\Delta \leq \frac{(M-1)}{2}\sum_{s=3}^M  (s-2)|Y_s|\leq (2n-9)n.\]
	This is impossible as $\Delta\leq 0$. 

	Lastly, we assume $M=2n+1$ and $Y_M=\{e\}$. This is possible only when $T=T^{(4)}$. 
	In that case, $T^{(4)}T=T^2=2G-T^{(2)}+2n$. It follows from \eqref{0mod5} that $ T^{(5)}=T^{(2)}$. 
	For any $t\in T$, there exists $s\in T$ such that $t^5=s^2$. As $T^{(4)}=T$, $t^4\in T$. Hence, $t^4 t=s^2$ and $s^2\in T^{(2)}$. By Lemma \ref{lm:TmeetT2}, this is possible only when $t=t^4=s$. But it then follows that $t^3=e$. Hence $|T|\leq 3$ which is impossible. 
\end{proof}

\medskip

\begin{proposition}\label{prop:n=1mod3}
	Theorem \ref{th:main} is true if $n\equiv 1 \pmod{3}$.
\end{proposition}
\begin{proof}
	By Proposition \ref{prop:nmod5=0}, we only have to consider the 4 cases when $n\not\equiv 0 \pmod{5}$.
	
	\textbf{(\rn{1})} $n\equiv 1\pmod{5}$: By \eqref{eq:T^4T} and Lemma \ref{lm:n=1mod3},
	\begin{align*}
		T^{(4)}T &\equiv 4G+T^{(2)}T+T-T^{(5)} \pmod{5}\\
				 &\equiv 4G+X_1 + 2 X_2 + 3X_3+T-T^{(5)} \pmod{5}.
	\end{align*}
	As $e\in T, T^{(5)}, X_1, G$ and $e\not\in X_2, X_3$, the identity element $e$ appears in $Y_{5i}$ for some $i\geq 1$. In view of the above equation, we deduce that 	
	\[ e\notin X_3\setminus ( T\cup T^{(5)})\subset \bigcup_{i=1} Y_{5i+2}, \]
	\[ e\notin X_0\setminus ( T\cup T^{(5)})\subset \bigcup_{i=1} Y_{5i+4}, \mbox{ and } 
	 e\notin X_2\setminus ( T\cup T^{(5)})\subset \bigcup_{i=1} Y_{5i+1}.\]
	In view of \eqref{eq:Y_is_condition2}, we get 	
		\begin{align*}
		(2n+1)^2 &\geq  \sum_{i=0}5i|Y_{5i}| +\sum_{i=0}2|Y_{5i+2}| + \sum_{i=0}4|Y_{5i+4}|+\sum_{i=0}|Y_{5i+1}|\\
				 &\geq 5  + \sum_{i=0}2|Y_{5i+2}| + \sum_{i=0}4|Y_{5i+4}|+\sum_{i=0}|Y_{5i+1}|\\
				 &\geq 5 + 2|X_3\setminus ( T\cup T^{(5)})|+4|X_0\setminus ( T\cup T^{(5)})| + |X_2\setminus ( T\cup T^{(5)})| \\
				 &\geq 5 + 2|X_3| + 4|X_0| + |X_2|-4|(T\cup T^{(5)})\setminus \{e\}|\\
				 &\geq 5+ \frac{16n^2-52n}{3} \qquad \text{(by Lemma \ref{lm:n=1mod3}),}
	\end{align*}
	where the second last inequality comes from the fact that the number of elements of $T\cup T^{(5)}$ in the disjoint union of $X_0, X_2$ and $X_3$ is at most the size of $(T\cup T^{(5)})\setminus\{e\}$.
	
	This means $4n^2-64n+12\leq 0$ whence $n\leq 15$. However, according to Lemma \ref{lm:list<=100}, this is impossible.
	
	\textbf{(\rn{2})} $n\equiv 2\pmod{5}$: 
	\begin{align*}
	T^{(4)}T &\equiv 2T^{(2)}T-T^{(5)} \pmod{5}\\
	&\equiv 0X_0+2X_1 + 4 X_2 + X_3-T^{(5)} \pmod{5}.
	\end{align*}
	Recall that $X_0$, $X_1$, $X_2$ and $X_3$ form a partition of $G$ and all nonzero coefficients in $T^{(5)}$ are $1$. Therefore, 
 	\[ Y_1  \setminus T^{(5)}\subset X_3 , Y_2\setminus T^{(5)}\subset X_1,  Y_3\setminus T^{(5)}=\emptyset  \mbox{ and } Y_4\setminus T^{(5)}\subset X_2.\]
	It follows that $|Y_1| \leq |X_3|+x_1$,  $|Y_2|\leq |X_1|+x_2$,  $|Y_3|\leq x_3$ and $|Y_4|\leq |X_2|+x_3$ where $x_1+x_2+x_3+x_4\leq |T^{(5)}|=2n+1$. Hence, from \eqref{eq:value_Xi_nmod3=1} we can derive
	\begin{equation}\label{eq:2Y1+3Y2+3Y3+2Y4_upperbound_2}
		2|Y_1| + 3|Y_2| +3|Y_3| +2 |Y_4| \leq \frac{8n(n-1)}{3}+3+8n+3(2n+1)=\frac{8}{3}n^2+\frac{34}{3}n+6.
	\end{equation}
	On the other hand, it follows from \eqref{ok1} and \eqref{eq:2Y1+3Y2+3Y3+2Y4_upperbound_2}, we have
	\[\frac{8}{3}n^2+\frac{34}{3}n+6 \geq4n^2+6n+2,\]
	which means $n\leq 6$. Hence $n=2$. However, it contradicts the assumption that $n\equiv 1 \pmod{3}$.
	
	\textbf{(\rn{3})} $n\equiv 3\pmod{5}$: By \eqref{eq:T^4T} and Lemma \ref{lm:n=1mod3},
	\begin{align*}
	T^{(4)}T &\equiv G - 2T^{(2)}T+2T-T^{(5)} \pmod{5}\\
	&\equiv G+ 0X_0+3X_1 +  X_2 + 4X_3 + 2T-T^{(5)} \pmod{5}\\
	&\equiv 1X_0+4X_1 +  2X_2 + 0X_3 + 2T-T^{(5)} \pmod{5}.
	\end{align*}
	In this case, $5$ divides $|G|$ and it is not necessarily true that all nonzero coefficients in $T^{(5)}$ are $1$. One may write $T^{(5)}=\sum_{i=1}^{k} i Z_i $ where $\bigcup_{i=1}^{k} Z_i=T^{(5)}$ and
	$\sum_{i=1}^{k}  i|Z_i|=2n+1$. However, as we will see below, we can still get some contradiction by checking the bounds for $|Y_1|, |Y_2|, |Y_3|$ and $|Y_4|$ as before. 
	Observe that 
	\[ Y_1\backslash (T\cup T^{(5)})\subset X_0,  Y_2\backslash (T\cup T^{(5)})\subset X_2, \]
	\[ Y_3\backslash (T\cup T^{(5)})=\emptyset \mbox{ and } Y_4\backslash (T\cup T^{(5)})=\emptyset.\]
	Note that the last equation is true because $X_1=\{e\}$.
	\begin{align*}
			&2|Y_1| + 3|Y_2| +3|Y_3| +2 |Y_4|\\
			\leq& 2|X_0|+3|X_2|+3|T\cup T^{(5)}|\\
			 \leq& \frac{4n(n-1)}{3} +12n+12n +3 \qquad \text{(by Lemma \ref{lm:n=1mod3})}\\
			 =& \frac{4n^2}{3} +\frac{68n}{3}+3.
	\end{align*}
	By \eqref{ok1}, we have
	\[\frac{4n^2}{3} +\frac{68n}{3} +3 \geq4n^2+6n+2,\]
	which implies $n\leq 6$. Taking account of the values of $n$ modulo $3$ and $5$, we see that $n=13$.  However, according to Lemma \ref{lm:list<=100}, there is no lattice tiling of $\Z^{13}$ by $S(13,2)$.
	
	\textbf{(\rn{4})} $n\equiv 4\pmod{5}$: By \eqref{eq:T^4T} and Lemma \ref{lm:n=1mod3},
	\begin{align*}
	T^{(4)}T &\equiv 3G - T^{(2)}T+2T-T^{(5)} \pmod{5}\\
	&\equiv 3G + 0X_0+4X_1 + 3X_2 + 2X_3 + 2T-T^{(5)} \pmod{5}\\
	&\equiv 3X_0+2X_1 + X_2 + 0X_3 + 2T-T^{(5)} \pmod{5}.
	\end{align*}
	As before, we obtain
	\[ Y_1\backslash (T\cup T^{(5)})\subset X_2,  Y_2\backslash (T\cup T^{(5)})=\emptyset, \]
	\[  Y_3\backslash (T\cup T^{(5)})\subset X_0 \mbox{ and } Y_4\backslash (T\cup T^{(5)})=\emptyset.\]
	Together with \eqref{eq:value_Xi_nmod3=1} we get
	\begin{align*}
				&2|Y_1| + 3|Y_2| +3|Y_3| +2 |Y_4| \\
				\leq &2|X_2|+3|X_1|+3|X_0|+3|T\cup T^{(5)}|\\
				\leq &8n+ 3\frac{2n(n-1)}{3} +12n+3\\
				 =&2n^2+18n+3.	
	\end{align*}	
	By \eqref{ok1}, 
	\[ 4n^2+6n+2\leq 2n^2+18n+3.\]
	Hence $n\leq 6$ which means $n=4$. However, this value has been already excluded by Lemma \ref{lm:list<=100}.
\end{proof}

\medskip

\begin{proposition}\label{prop:n=2mod3}
	Theorem \ref{th:main} is true for $n\equiv 2 \pmod{3}$.
\end{proposition}
	\begin{proof} 
	By Proposition \ref{prop:nmod5=0}, we only have to investigate the 4 cases when $n\not\equiv 0 \pmod{5}$.
	
	\textbf{(\rn{1})} $n\equiv 1\pmod{5}$: By \eqref{eq:T^4T} and Lemma \ref{lm:n=2mod3},
	\begin{align*}
		T^{(4)}T &\equiv 4G+T^{(2)}T+T-T^{(5)} \pmod{5}\\
				 &\equiv 4G+X_1 +  2X_2 + 3X_3 + 4X_4+T-T^{(5)} \pmod{5}.
	\end{align*}
	Thus,  	
	\[ Y_1\backslash (T\cup T^{5)})\subset X_2, Y_2\backslash (T\cup T^{5)})\subset X_3, 
	Y_3\backslash (T\cup T^{5)})\subset X_4, Y_4\backslash (T\cup T^{5)})=\emptyset.\]
	Hence, 
	\begin{align*}
		&2|Y_1| + 3|Y_2| +3|Y_3| +2 |Y_4|\\ 
		\leq& 2|X_2|+3|X_2|+3|X_4|+3|T\cup T^{(5)}|\\
		\leq& 10n+(2n^2-4n)+12n+3\qquad (\text{By Lemma \ref{lm:n=2mod3}})\\
		=& 2n^2+18n+3.
	\end{align*}
	Therefore, $4n^2+6n+2\leq 2n^2+18n+3$ which means $n\leq 6$. By Lemma \ref{lm:list<=100}, this is impossible. 
	
	Alternatively, as $5\mid 2n^2+2n+1$, $8n-3\neq 5k^2$ for some $k\in \Z$, and $8n+1$ is not a square, one may also use \cite[Corollary 3.9 (1)]{zhang_nonexistence_2019} to prove this case.
	
	\textbf{(\rn{2})} $n\equiv 2\pmod{5}$: By \eqref{eq:T^4T} and Lemma \ref{lm:n=2mod3},
	\begin{align*}
	T^{(4)}T &\equiv 2T^{(2)}T-T^{(5)} \pmod{5}\\
	&\equiv 2X_1 + 4 X_2 + X_3 + 3X_4-T^{(5)} \pmod{5}.
	\end{align*}
	It is easy to see that 
	\[ X_1\backslash T^{(5)}\subset \bigcup_{i=0} Y_{5i+2} \mbox{ and } 
	X_4\backslash T^{(5)}\subset \bigcup_{i=0} Y_{5i+3}.\]
	It follows that 
	\[\sum_{i=0}(5i+2)|Y_{5i+2}| \geq 2|X_1|-2x \mbox{ and }  \sum_{i=0}(5i+3)|Y_{5i+3}| \geq 3 |X_4|-3y\]
	with $x+y\leq 2n+1$. 
	Thus, by \eqref{eq:value_Xi_nmod3=2}
	\begin{equation*}
	(2n+1)^2=\sum_{i=1}^Mi|Y_i|\geq 2|X_1|+3|X_4|-6n-3=\frac{14n^2-34n}{3}-1.
	\end{equation*}
	By calculation, we get $2n^2-46n-6\leq0$ which implies that $n\leq 23$. As $n$ is congruent to $2 \mod{5}$, $n\neq 16, 21$. Therefore by Lemma \ref{lm:list<=100}, we have a contradiction.
	
	\textbf{(\rn{3})} $n\equiv 3\pmod{5}$: By \eqref{eq:T^4T} and Lemma \ref{lm:n=2mod3},
	\begin{align*}
		T^{(4)}T &\equiv G - 2T^{(2)}T+2T-T^{(5)} \pmod{5}\\
				&\equiv G+ 3X_1 +  X_2 + 4X_3 + 2X_4 + 2T-T^{(5)} \pmod{5}\\
				&\equiv 4X_1 +  2X_2 + 0X_3 +3X_4 + 2T-T^{(5)} \pmod{5}.
	\end{align*}
	This implies that
	\[ X_1\backslash (T\cup T^{(5)})\subset \bigcup_{i=0} Y_{5i+4} \mbox{ and } 
	X_4\backslash (T\cup T^{(5)})\subset \bigcup_{i=0} Y_{5i+3}.\]
	Therefore, 
	\[\sum_{i=0}|Y_{5i+4}| \geq |X_1|-x,\]
	and 
	\[\sum_{i=0}|Y_{5i+3}| \geq  |X_4|-y\]
	where $0\leq x+y\leq 4n+1$. 
	Hence, by \eqref{eq:value_Xi_nmod3=2}
	\begin{equation*}
		(2n+1)^2=\sum_{i=1}^Mi|Y_i|\geq 4|X_1|+3|X_4|-16n -4=\frac{22n^2-68n}{3},
	\end{equation*}
	which implies that $10n^2-74n-3\leq 0$ whence $n\leq 7$. According to Lemma \ref{lm:list<=100}, there is no such a lattice tiling of $\Z^n$ by $S(n,2)$.
	
	\textbf{(\rn{4})} $n\equiv 4\pmod{5}$: By \eqref{eq:T^4T} and Lemma \ref{lm:n=2mod3},
	\begin{align*}
	T^{(4)}T &\equiv 3G - T^{(2)}T+2T-T^{(5)} \pmod{5}\\
	&\equiv 3G + 4X_1 + 3X_2 + 2X_3 + X_4 + 2T-T^{(5)} \pmod{5}\\
	&\equiv 2X_1 + X_2 + 0X_3 + 4X_4 + 2T-T^{(5)} \pmod{5}.
	\end{align*}
	As before, we obtain 
	\[\sum_{i=0}|Y_{5i+2}| \geq |X_1|-x,\]
	and 
	\[\sum_{i=0}|Y_{5i+4}| \geq  |X_4|-y\]
	where $0\leq x+y\leq 4n+1$. 
	Hence, by \eqref{eq:value_Xi_nmod3=2}
	\begin{equation*}
	(2n+1)^2=\sum_{i=1}^Mi|Y_i|\geq 2|X_1|+4|X_4|-16n-4=\frac{16n^2-68n}{3}-2,
	\end{equation*}
	which implies that $4n^2-74n-3\leq 0$ whence $n\leq 18$. As $n$ is congruent to $4 \mod 5$, $n\neq 16$. Thus by Lemma \ref{lm:list<=100}, this is impossible.
\end{proof}
\begin{proof}[Proof of Theorem \ref{th:main}]
	Propositions \ref{prop:n=0mod3}, \ref{prop:n=1mod3} and \ref{prop:n=2mod3} together form a complete proof of Theorem \ref{th:main}.
\end{proof}

\begin{remark}\label{remark:no_inf}
	 Although our main result shows that there is no lattice tiling of $\mathbb{Z}^n$ by $S(n, 2)$ for all $n\geq 3$; we still do not know whether Golomb-Welch conjecture has been proved in the case of $r = 2$ for infinitely many values of $n$. The reason is that we do not know whether $f(n) = 2n^2 +2n+1$ is a prime for infinitely many values of $n$. A positive answer to this question would solve a special case of the famous conjecture of Bunyakovsky (1857) that asks whether there exists an irreducible quadratic polynomial attaining a prime number value infinitely many times.
\end{remark}

\section*{Acknowledgment}
The authors thank the referees for their helpful comments and suggestions, which improve the presentation of this paper. The first author is supported by grantR-146-000-158-112, Ministry of Education, Singapore. The second author is supported by the National Natural Science Foundation of China (No.\ 11771451) and Natural Science Foundation of Hunan Province (No.\ 2019RS2031).


\begin{thebibliography}{10}
	
	\bibitem{bannai_1973_moore}
	E.~Bannai and T.~Ito.
	\newblock On finite {Moore} graphs.
	\newblock {\em J. Fac. Sci. Tokyo}, 20:191--208, 1973.
	
	\bibitem{camarero_quasi-perfect_lee_2016}
	C.~Camarero and C.~Mart\'inez.
	\newblock Quasi-perfect {Lee} codes of radius 2 and arbitrarily large
	dimension.
	\newblock {\em IEEE Transactions on Information Theory}, 62(3):1183--1192,
	March 2016.
	
	\bibitem{cohen_sphere_packing_2003}
	H.~Cohen and N.~Elkies.
	\newblock New upper bounds on sphere packings {I}.
	\newblock {\em Annals of Mathematics}, 157:689{\textendash}714, 2003.
	
	\bibitem{damerell_moore_1973}
	R.~M. Damerell.
	\newblock On {Moore} graphs.
	\newblock {\em Mathematical Proceedings of the Cambridge Philosophical
		Society}, 74:227--236, 1973.
	
	\bibitem{dougherty_degree-diameter_2004}
	R.~Dougherty and V.~Faber.
	\newblock The degree-diameter problem for several varieties of cayley graphs
	{I}: {The} abelian case.
	\newblock {\em SIAM Journal on Discrete Mathematics}, 17(3):478--519, Jan.
	2004.
	
	\bibitem{golomb_perfect_1970}
	S.~W. Golomb and L.~R. Welch.
	\newblock Perfect codes in the {Lee} metric and the packing of polyominoes.
	\newblock {\em SIAM Journal on Applied Mathematics}, 18(2):302--317, 1970.
	
	\bibitem{hoffman_moore_1960}
	A.~Hoffman and R.~Singleton.
	\newblock On {Moore} graphs with diameters 2 and 3.
	\newblock {\em IBM Journal of Research and Development}, 4(5):497--504, Nov.
	1960.
	
	\bibitem{horak_tilings_2009}
	P.~Horak.
	\newblock Tilings in {Lee} metric.
	\newblock {\em European Journal of Combinatorics}, 30(2):480--489, Feb. 2009.
	
	\bibitem{horak_new_2014}
	P.~Horak and O.~Gro\v{s}ek.
	\newblock A new approach towards the {Golomb}-{Welch} conjecture.
	\newblock {\em European Journal of Combinatorics}, 38:12--22, May 2014.
	
	\bibitem{horak_50_2018}
	P.~Horak and D.~Kim.
	\newblock 50 years of the {Golomb-Welch} conjecture.
	\newblock {\em IEEE Transactions on Information Theory}, 64(4):3048--3061, Apr.
	2018.
	
	\bibitem{kari_algebraic_2015}
	J.~Kari and M.~Szabados.
	\newblock An algebraic geometric approach to {Nivat's} conjecture.
	\newblock In M.~M. Halld{\'o}rsson, K.~Iwama, N.~Kobayashi, and B.~Speckmann,
	editors, {\em Automata, Languages, and Programming}, pages 273--285, Berlin,
	Heidelberg, 2015. Springer Berlin Heidelberg.
	
	\bibitem{kim_2017_nonexistence}
	D.~Kim.
	\newblock Nonexistence of perfect 2-error-correcting {Lee} codes in certain
	dimensions.
	\newblock {\em European Journal of Combinatorics}, 63:1 -- 5, 2017.
	
	\bibitem{lepisto_modification_1981}
	T.~Lepist{\"o}.
	\newblock A modification of the elias-bound and nonexistence theorems for
	perfect codes in the {Lee}-metric.
	\newblock {\em Information and Control}, 49(2):109 -- 124, 1981.
	
	\bibitem{miller_moore_2013}
	M.~Miller and J.~\v{S}ir\'{a}\v{n}.
	\newblock Moore graphs and beyond: {A} survey of the degree/diameter problem.
	\newblock {\em The Electronic Journal of Combinatorics}, 20(2):DS14, May 2013.
	
	\bibitem{post_nonexistence_1975}
	K.~A. Post.
	\newblock Nonexistence theorems on perfect lee codes over large alphabets.
	\newblock {\em Information and Control}, 29(4):369 -- 380, 1975.
	
	\bibitem{qureshi_nonexistence_ZGcondition_arxiv}
	C.~Qureshi.
	\newblock On the non-existence of linear perfect {Lee} codes: The {Zhang-Ge}
	condition and a new polynomial criterion.
	\newblock {\em arXiv:1805.10409v1 [cs.IT]}, 2018.
	
	\bibitem{szegedy_algorithms_1998}
	M.~Szegedy.
	\newblock Algorithms to tile the infinite grid with finite clusters.
	\newblock In {\em Proceedings 39th {Annual} {Symposium} on {Foundations} of
		{Computer} {Science} ({Cat}. {No}.98CB36280)}, pages 137--145, Nov. 1998.
	
	\bibitem{zhang_perfect_lp_2007}
	T.~Zhang and G.~Ge.
	\newblock Perfect and quasi-perfect codes under the {$l_p$} metric.
	\newblock {\em IEEE Trans. Inform. Theory}, 63(7):4325--4331, 2017.
	
	\bibitem{zhang_nonexistence_2019}
	T.~Zhang and Y.~Zhou.
	\newblock On the nonexistence of lattice tilings of $\mathbb{Z}^n$ by {Lee}
	spheres.
	\newblock {\em Journal of Combinatorial Theory, Series A}, 165:225 -- 257,
	2019.
	
\end{thebibliography}
\end{document}